\documentclass[12pt]{amsart}
\usepackage{epsfig,color}

\makeatother

\theoremstyle{definition}

\newtheorem{thm}{Theorem}[section]
\newtheorem{cor}[thm]{Corollary}

\newtheorem{lem}[thm]{Lemma}

\newtheorem{rmk}[thm]{Remark}
\newtheorem{prop}[thm]{Proposition}

\numberwithin{equation}{section}
\address{MIT, Dept. of Math.\\
77 Massachusetts Avenue, Cambridge, MA 02139-4307.}
\email{jchang61@mit.edu}

\title{Lower bounds for nodal sets of biharmonic Steklov problems}
\author{Jui-En Chang} 
\date{\today}

\usepackage{latexsym, amsmath,amssymb}
\usepackage{amsmath}
\usepackage{amsfonts}
\usepackage{amssymb}
\usepackage{amscd}
\usepackage{amsthm}
\usepackage{amsbsy}
\usepackage{graphicx}
\usepackage{bm}
\usepackage{epsfig,epstopdf,url,array}%titling

\begin{document}
\maketitle
\begin{abstract}
We use layer potential to establish that the boundary biharmonic Steklov operators are elliptic pseudo-differential operators. Thus we are able to establish lower bounds on both the measure of boundary nodal sets and interior nodal sets for biharmonic Steklov eigenfunctions.
\end{abstract}

\section{Introduction}
Let $M$ be a compact $n$-dimensional manifold with smooth boundary $\partial M$. We consider the following three biharmonic Steklov eigenfunction problems:
\begin{equation} \label{Theta}
\Bigg \{
\begin{array}{ll}
\triangle^2 e_\lambda=0 & $in $ M \\
\partial_\nu e_\lambda=\partial_\nu \triangle e_\lambda+\lambda^3e_\lambda=0 & $on $ \partial M; \end{array}
\end{equation}
\begin{equation} \label{Xi}
\Bigg \{
\begin{array}{ll}
\triangle^2 e_\lambda=0 & $in $ M \\
e_\lambda=\triangle e_\lambda-\lambda\partial_\nu e_\lambda=0 & $on $ \partial M; \end{array}
\end{equation}
\begin{equation} \label{Pi}
\Bigg \{
\begin{array}{ll}
\triangle^2 e_\lambda=0 & $in $ M \\
e_\lambda=\partial_\nu^2 e_\lambda-\lambda\partial_\nu e_\lambda=0 & $on $ \partial M. \end{array}
\end{equation}
The problems arise in elastic mechanics. When the weight of the body $M$ is the only body force, the stress function must be biharmonic in $M$. In addition, the problem \eqref{Xi} is referred to as the Dirichlet eigenvalue problem in \cite{P} and it is related to the study of Poisson ratio in theory of elasticity, see \cite{FGW}. Kutter and Sigillito\cite{KS}, Payne\cite{P}, Wang and Xia\cite{WX} focus on giving bounds for the first eigenvalues, which are related to the geometry of the manifold.

These problems are also important in biharmonic analysis and the inverse problem. The related problem was initially studied by Calder\'{o}n\cite{C}. The connection is that the set of the eigenvalues for the biharmonic Steklov problem \eqref{Theta} and \eqref{Xi} are the same as that of the well-known ``Dirichlet to Neumann Laplacian" map and the ``Neumann to Laplacian" map for biharmonic equation, respectively. These maps concern the relation between different boundary data of biharmonic functions.

In each of the problems, the spectrum is discrete and the only accumulation points of eigenvalues is infinity. In view of the important applications, one is interested in finding the asymptotic behavior for eigenvalues and corresponding eigenfunctions. The Weyl-type formula which concerns the distribution of the eigenvalues is given by Liu in \cite{L1}, \cite{L2}. In this paper, we are interested in the behavior of the eigenfunctions. We give lower bounds of the measure of nodal sets for all the biharmonic Steklov problems.

Let us briefly review the literature concerning the study of the nodal sets for other eigenfunction problems. First, the eigenfunction $e_\lambda$ with eigenvalue $\lambda$ for the Laplace operator on a smooth compact manifold satisfies
\begin{equation}
\triangle e_\lambda=-\lambda^2 e_\lambda.
\end{equation}
Yau\cite{Y1}, \cite{Y2} conjectures that for the nodal set $Z_\lambda=\{e_\lambda=0\}$,
\begin{equation}
c\lambda\leq|Z_\lambda|\leq C\lambda
\end{equation}
for some constant $c,$ $C$ only depends on $M$. The lower bound was proven by Br\"{u}ning\cite{B} and Yau for surfaces. Donnelly and Fefferman\cite{DF} establish both bounds for analytic manifolds. However, for $n$-dimensional smooth manifolds, the conjecture remains open. In this case, there are some bounds established for the nodal sets which are weaker than that in the conjecture. Colding and Minicozzi\cite{CM} give the best lower bound $|Z_\lambda| \geq c\lambda^\frac{3-n}{2}$. For alternative proofs, see Hezari and Sogge\cite{HS}, Sogge and Zelditch\cite{SZ}.

For the case of the harmonic Steklov eigenfunction, $e_\lambda$ is defined as the solution of 
\begin{equation} 
\Bigg \{
\begin{array}{cc}
\triangle e_\lambda=0 & $in $ M \\
\partial_\nu e_\lambda=\lambda e_\lambda & $on $ \partial M. \end{array}
\end{equation}
If we restrict $e_\lambda$ on the boundary, it satisfies the eigenvalue problem
\begin{equation}
\Lambda e_\lambda=\lambda e_\lambda,
\end{equation}
where $\Lambda$, the Dirichlet-to-Neumann operator, is defined as
\begin{equation}
\Lambda f=\partial_\nu(Kf)|_{\partial M},
\end{equation}
for $f\in C^\infty(\partial M)$ and $Kf$ is the unique harmonic function in $M$ with boundary value $f$.

Let $\tilde{Z}_\lambda=\{x\in\partial M|e_\lambda=0\}$ be the boundary nodal set for harmonic Steklov eigenfunction $e_\lambda$. Bellova and Lin\cite{BL} first establish a upper bound $|\hat{Z_\lambda}|\leq C\lambda^6$ on analytic domains in $\mathbb{R}^n$. Later, Zelditch\cite{Z} gives the sharp upper bound $|\hat{Z_\lambda}|\leq C\lambda$ on analytic manifolds with analytic boundary. For the case of a smooth manifold, Wang and Zhu\cite{WZ} establish a lower bound $|\tilde{Z}_\lambda|\geq c \lambda^\frac{4-n}{2}$. Notice that this is the same order as in the Laplacian operator case since the dimension of $\partial M$ is $n-1$. For the interior nodal set $Z_\lambda=\{e_\lambda=0\}\subset M$, Sogge, Wang and Zhu\cite{SWZ} establish a lower bound $|Z_\lambda|\geq c\lambda^\frac{2-n}{2}$. All the current best lower bounds on smooth manifolds employ the theory of pseudo-differential operators to get an $L^p$ estimate of eigenfunctions.

In this article, we establish polynomial lower bounds for size of the boundary nodal sets, the vanishing sets of $\triangle e_\lambda$ and the interior nodal sets.
\begin{thm}\label{boundaryset}
If 0 is a regular value of $e_\lambda$ on $\partial M$ for \eqref{Theta} case, or 0 is a regular value of $\partial_\nu e_\lambda$ on $\partial M$ for \eqref{Xi}, \eqref{Pi} case, we have
\begin{equation}
\big|\tilde{Z}_\lambda\big|\geq c\lambda^{\frac{4-n}{2}}.
\end{equation}
where
\begin{equation}
\begin{array}{ll}
\tilde{Z}_\lambda=\{x\in\partial M|e_\lambda=0\} & $for problem \eqref{Theta}$,\\
\tilde{Z}_\lambda=\{x\in\partial M|\partial_\nu e_\lambda=0\} & $for problems$ \eqref{Xi}. \eqref{Pi}.\end{array}
\end{equation}
\end{thm}
\begin{thm}\label{innerset}
For $e_\lambda$ satisfying \eqref{Theta}, \eqref{Xi} or \eqref{Pi}, if 0 is a regular value of $\triangle e_\lambda$, we have
\begin{equation}
\big|\hat{Z}_\lambda\big|\geq c\lambda^{\frac{2-n}{2}},
\end{equation}
where
\begin{equation}
\hat{Z}_\lambda=\{x\in M|\triangle e_\lambda=0\}.
\end{equation}
\end{thm}

\begin{thm}\label{realinnerset}
For $e_\lambda$ satisfying \eqref{Xi} or \eqref{Pi}, if 0 is a regular value of $e_\lambda$, we have
\begin{equation}
\big|Z_\lambda\big|\geq c\lambda^{\frac{2-n}{2}}.
\end{equation}
For $e_\lambda$ satisfying \eqref{Theta}, if 0 is a regular value of $e_\lambda$, we have
\begin{equation}
\big|Z_\lambda\big|\geq c\lambda^{\frac{-n}{2}},
\end{equation}
where
\begin{equation}
Z_\lambda=\{x\in M|e_\lambda=0\}.
\end{equation}
\end{thm}

\begin{rmk}
The reader may compare theorem \ref{realinnerset} to the lower bound of interior Steklov nodal sets given in \cite{SWZ}, which is of the same order except the bound for $\big|Z_\lambda\big|$ in \eqref{Theta} case. Also, the reader may compare theorem \ref{boundaryset} to the lower bound of boundary Steklov nodal sets given in \cite{WZ}. Again, we get a lower bound with the same order.
\end{rmk}

The rest of this article is organized as follows. In section 2, we introduce related boundary operators and establish important equations for biharmonic functions. In section 3, using the method of layer potentials as in \cite{T}, we show that the boundary operators are elliptic pseudo-differential operators on $\partial M$, which is different from the proof given in \cite{L2}. By the pseudo-differential operator theory, we establish the $L^p$ estimates from the theorem in \cite{S}. From this, $L^\infty$, $L^2$, $L^1$ bounds for $|\nabla e_\lambda|$ is given on the sets which we want to find a lower bound as in \cite{WZ} and \cite{HS}. In section 4, 5 and 6, we focus on the set $\hat{Z}_\lambda$, $Z_\lambda$ and $\tilde{Z}_\lambda$ respectively and prove the theorems.

\section{Some basic properties for the biharmonic Steklov problem}

The biharmonic Steklov problems are related to the boundary operators. The eigenfunctions $e_\lambda$ in \eqref{Theta}, \eqref{Xi} and \eqref{Pi} satisfy the eigenvalue problems
\begin{equation}
\begin{split}
\Theta e_\lambda&=\lambda^3 e_\lambda,\Xi\partial_\nu e_\lambda=\lambda \partial_\nu e_\lambda,\Pi\partial_\nu e_\lambda=\lambda \partial_\nu e_\lambda,\\
\end{split}
\end{equation}
on $\partial M$, respectively, for the Dirichlet-to-Neumann-Laplacian operator $\Theta$, Neumann-to-Laplacian operator $\Xi$, Neumann-to-double-Neumann operator $\Pi$, which are defined below. For $f\in C^\infty(\partial M)$, define 
\begin{equation}
\begin{split}
\Theta f&=-\partial_\nu\triangle(K_1f)|_{\partial M},\\
\Xi f&=\triangle(K_2f)|_{\partial M},\\
\Pi f&=\partial_\nu^2(K_2f)|_{\partial M},\\
\end{split}
\end{equation}
where $K_1f=u$ is the unique biharmonic function with $u|_{\partial M}=f$, $\partial_\nu u|_{\partial M}=0$ and $K_2f=v$ is the unique biharmonic function with $v|_{\partial M}=0$, $\partial_\nu v|_{\partial M}=f$.

First, let us show the operator $\Pi$ is well defined and establish the relation between $\Xi$ and $\Pi$.
\begin{thm}
The $\partial_\nu^2$ in the definition of $\Pi$ is well defined. We have $\Xi f=\Pi f+Hf$, where $H$ is the mean curvature of $\partial M$.
\end{thm}
\begin{proof}
Let $F$ be a smooth function on $M$ with $F|_{\partial M}=0$. Let $N$ be any unit vector field defined in a neighborhood of $\partial M$ with $N|_{\partial M}=\partial_\nu$. We have
\begin{equation}
NNF=N(dF(N))=\nabla^2F(N,N)+dF(\nabla_NN).
\end{equation}
Since $\nabla_NN\perp N$, $\nabla_NN|_{\partial M}$ is tangent to $\partial M$. Using $F|_{\partial M}=0$, we have $dF(\nabla_NN)=0$. Therefore, $NNF=\nabla^2F(N,N)$ is tensorial and only depends on $N|_{\partial M}=\partial_\nu$.

Now, let $\{e_i\}_{i=1}^{n-1}\cup\{N\}$ be an orthonormal frame in a neighborhood of the boundary. We have
\begin{equation}
\triangle F=\nabla^2F(N,N)+\sum_{i=1}^{n-1}
\nabla^2F(e_i,e_i)=\nabla^2F(N,N)+
\sum_{i=1}^{n-1}\Big(e_ie_iF-dF(\nabla_{e_i}e_i)\Big).
\end{equation}
Given that $F|_{\partial M}=0$, $e_i|_{\partial M}$ is tangent to $\partial M$, we have $e_ie_iF=0$.
\begin{equation}
\triangle F|_{\partial M}=\partial_\nu^2F-
\sum_{i=1}^{n-1}dF(\nabla_{e_i}e_i)=\partial_\nu^2F-
dF(\sum_{i=1}^{n-1}\nabla_{e_i}e_i)=\partial_\nu^2F+H\partial_\nu F.
\end{equation}
Plug in $F=K_2f$ and we can get the desired result.
\end{proof}

Next, let us recall the Green's formula for biharmonic function:
\begin{thm}
Let $u$, $v$ be biharmonic function on $M$, we have
\begin{equation}
\begin{split}
0&=\int_{\partial M}\partial_\nu\triangle uv-\triangle u\partial_\nu v+\partial_\nu u \triangle v-u\partial_\nu\triangle v,\\
\int_M(\triangle u)^2&=\int_{\partial M}\partial_\nu u \triangle u-u\partial_\nu\triangle u.\\
\end{split}
\end{equation}
\end{thm}
\begin{proof}
Integrate the equation
\begin{equation}
\begin{split}
\triangle^2uv-u\triangle^2v&=\nabla\cdot(\nabla\triangle uv-\triangle u\nabla v+\nabla u\triangle v-u\nabla\triangle v),\\
(\triangle u)^2-u\triangle^2u&=\nabla\cdot(\nabla u\triangle u-u\nabla\triangle u).
\end{split}
\end{equation}
on $M$ and use divergence theorem.
\end{proof}
From the equation above, we can deduce that $\Theta$, $\Xi$, $\Pi$ are self-adjoint, $\Theta$ and $\Xi$ are positive.

Now, let $\hat{E}(x,y)$ be a symmetric fundamental solution to the biharmonic equation:
\begin{equation}
\triangle^2_x\hat{E}(x,y)=\delta_y(x).
\end{equation}
From the symmetry, we also have $\triangle^2_y\hat{E}(x,y)=\delta_x(y)$. We have the following Green's formula for the biharmonic functions.
\begin{thm}
If $u$ is a biharmonic function on $M$, we have
\begin{equation}
u(x)=\int_{\partial M}\Big[u\cdot\partial_\nu^y\triangle_y\hat{E}(x,y)-\partial_\nu u\cdot\triangle_y\hat{E}(x,y)+\triangle u\cdot\partial_\nu^y \hat{E}(x,y)-\partial_\nu\triangle u\cdot\hat{E}(x,y) \Big]d\sigma(y).
\end{equation}
\end{thm}

\section{Layer potentials}
Now, to establish the result, we need some technical results for biharmonic boundary Steklov operators. Let $M$ as above, we can extend the manifold across the boundary such that $\bar{M}\subset\Omega$, where $\Omega$ is a smooth $n$-dimensional manifold. Let $\mathcal{O}\subset\Omega$ be a precompact open neighborhood of $\bar{M}$. Start with a symmetric fundamental solution $E^\circ(x,y)$ of the Laplacian operator,
\begin{equation}
\triangle_x E^\circ(x,y)=\delta_y(x),
\end{equation}
where $E^\circ(x,y)$ is the Schwartz kernel of the operator $E^\circ(x,D)\in OPS^{-2}(\Omega)$. Now, let $\eta\in C^\infty_0(\Omega)$ be a cutoff function which is identically $1$ in $\mathcal{O}$ and $E(x,y)=\eta(x)\eta(y)E^\circ(x,y)$ be the Schwartz kernel of the compactly supported operator $E(x,D)\in OPS^{-2}(\Omega)$. We can construct the following fundamental solution for the biharmonic equation:
\begin{equation}
\hat{E}(x,y)=\int_\Omega E(x,z)E(z,y)dV(z),
\end{equation}
which satisfies
\begin{equation}
\triangle^2_x\hat{E}(x,y)=\delta_y(x),  \triangle^2_y\hat{E}(x,y)=\delta_x(y)
\end{equation}
in $\mathcal{O}$. $\hat{E}(x,y)$ is the Schwartz kernel of a compactly supported operator $\hat{E}(x,D)$, where $\hat{E}(x,D)=E(x,D)E(x,D)\in OPS^{-4}(\Omega)$. The Schwartz kernel $\hat{E}(x,y)$ is smooth off the diagonal. As $x\to y$, we have the following expansion:
\begin{equation}
\hat{E}(x,y)= c_n\Bigg\{
\begin{array}{ll}
R(x,y)+$d$(x,y)^2\log($d$(x,y))+\cdots & n=2, \\
R(x,y)+$d$(x,y)+\cdots & n=3,\\
\log($d$(x,y))+\cdots & n=4,\\
$d$(x,y)^{4-n}+\cdots & n\geq 5,\end{array}
\end{equation}
where $R(x,y)$ is smooth, in dimension $n=2$, 3, they are more significant than the part contribute to $\triangle^2_x\hat{E}(x,y)=\delta_y(x)$, but they only contribute to a smoothing operator. The function d$(x,y)$ is the distance on the manifold, and the constant
\begin{equation}
c_n=\Bigg\{
\begin{array}{ll}
\frac{1}{4\omega_1}=\frac{1}{8\pi} & n=2, \\
\frac{-1}{4\omega_3} & n=4,\\
\frac{1}{2(4-n)(2-n)\omega_{n-1}} & n=3,$ $n\geq 5,\end{array}
\end{equation}
with $\omega_n=Vol(\mathbb{S}^n)$.
For a function $f$ on $\partial M$, follow the same approach as in \cite{T}, we define the following potentials in $M$.
\begin{equation}
\begin{array}{ll}
L_1f(x)=\int_{\partial M}f(y)\hat{E}(x,y)d\sigma(y), & L_2f(x)=\int_{\partial M} f(y)\partial_{\nu,y}\hat{E}(x,y)d\sigma(y),\\
L_3f(x)=\int_{\partial M} f(y)\triangle_y\hat{E}(x,y)d\sigma(y), & L_4f(x)=\int_{\partial M} f(y)\partial_{\nu,y}\triangle_y\hat{E}(x,y)d\sigma(y).
\end{array}
\end{equation}

Given a function $u$ on $\Omega\backslash\partial M$. For $x\in\partial M$, define $u_+(x)$ and $u_-(x)$ be the limit of $u(z)$ as $z\to x$, from $z\in M$ and $z\in \Omega\backslash\bar{M}$, respectively. Now, we can find the limit of the above layer potentials on $\partial M$.
\begin{prop}
For $x\in\partial M$, we have
\begin{equation} \label{s3}
L_1f_+(x)=L_1f_-(x)=S_3f(x),
\end{equation}
\begin{equation} \label{s2}
L_2f_+(x)=L_2f_-(x)=S_2f(x),
\end{equation}
\begin{equation} \label{s1}
L_3f_+(x)=L_3f_-(x)=S_1f(x),
\end{equation}
\begin{equation} \label{n}
L_4f_\pm(x)=\pm\frac{1}{2}f(x)+Nf(x),
\end{equation}
where, for $x\in\partial M$,
\begin{equation}
S_3f(x)=\int_{\partial M}f(y)\hat{E}(x,y)d\sigma(y),
\end{equation}
\begin{equation}
S_2f(x)=\int_{\partial M} f(y)\partial_{\nu,y}\hat{E}(x,y)d\sigma(y),
\end{equation}
\begin{equation}
S_1f(x)=\int_{\partial M} f(y)\triangle_y\hat{E}(x,y)d\sigma(y),
\end{equation}
and
\begin{equation}
Nf(x)=\int_{\partial M} f(y)\partial_{\nu,y}\triangle_y\hat{E}(x,y)d\sigma(y).
\end{equation}
Furthermore, for the operators defined above, we have $S_2$, $S_3\in OPS^{-3}(\partial M)$, $S_1$, $N\in OPS^{-1}(\partial M)$. $S_1$ and $S_3$ are elliptic. The principle symbols of $S_1$ and $S_3$ are the same as that of $\frac{1}{2}\sqrt{-\triangle^T}^{-1}$ and $\frac{1}{4}\sqrt{-\triangle^T}^{-3}$, respectively, where $\triangle^T$ is the Laplacian operator on the boundary.
\end{prop}
\begin{proof}
Following \cite{T}, if $\sigma\in\mathcal{E}'(\Omega)$ is the surface measure on $\partial M$, $f\in\mathcal{D}'(\partial M)$, we have $f\sigma\in\mathcal{E}'(\Omega)$. Now, let $p(x,D)\in OPS^m(\Omega)$, define 
\begin{equation}
v=p(x,D)(f\sigma).
\end{equation}
When $m<-1$, $v$ is continuous even across $\partial M$ and 
\begin{equation}
\begin{array}{ll}
v|_{\partial M}=Qf, & Q\in OPS^{m+1}(\partial M).
\end{array}
\end{equation}
We need to compute the principle symbols of them. At any point on $\partial M$, choose the coordinates such that $\{x_i\}_{i=1}^{n-1}$ are normal coordinates on $\partial M$ and $x_n$ is the normal direction pointing into $M$. The symbol of $Q(x,D)$ is given by
\begin{equation}
q(x_n,x',\xi')=\frac{1}{2\pi}\int p(x,\xi',\xi_n)e^{ix_n\xi_n}d\xi_n.
\end{equation}
In this coordinate, put $p_3(x,D)=\hat{E}(x,D)$, $p_2(x,D)=\hat{E}(x,D)X^*$, $p_1(x,D)=\hat{E}(x,D)\triangle$ respectively, where 
$X$ is any vector field on $\Omega$ which equals the outer normal $\nu$ on $\partial M$ and $X^*$ its formal adjoint. The corresponding principle symbols are  $p_3(x,\xi)=|\xi|^{-4}$, $p_2(x,\xi)=i\xi_n|\xi|^{-4}$, $p_1(x,\xi)=-|\xi|^2$. Use this, we can get
\begin{equation}
\begin{split}
q_3(x_n,x',\xi')&=\frac{1}{2\pi}\int^\infty_{-\infty}|\xi|^{-4}e^{ix_n\xi_n}d\xi_n=\frac{1}{4}(\frac{1}{|\xi'|^3}+\frac{|x_n|}{|\xi'|^2})e^{-|x_n\xi'|},\\
q_2(x_n,x',\xi')&=\frac{i}{2\pi}\int^\infty_{-\infty}\xi_n|\xi|^{-4}e^{ix_n\xi_n}d\xi_n=\frac{-x_n}{4|\xi'|}e^{-|x_n\xi'|},\\
q_1(x_n,x',\xi')&=\frac{-1}{2\pi}\int^\infty_{-\infty}|\xi|^{-2}e^{ix_n\xi_n}d\xi_n=-\frac{e^{-|x_n\xi'|}}{2|\xi'|},\\
\end{split}
\end{equation} 
Taking the limit as $x_n$ goes to 0. For $|\xi'|>1$ the right hand side uniformly converge. Therefore, after restricting on $\partial M$, the principle symbol of $S_3$, $S_1$ are $\frac{1}{4}|\xi'|^{-3}$, $-\frac{1}{2}|\xi'|^{-1}$ respectively. For $q_2$, since the right side converge to 0, and the term with $O(|\xi|^{-4})$ only contribute to $OPS ^{-3}(\partial M)$, we can conclude the resulting operator $S_2\in OPS ^{-3}(\partial M)$. We can establish \eqref{s3}, \eqref{s2}, \eqref{s1} and the properties of $S_3$, $S_2$, $S_1$.

Now, let us turn out attention to \eqref{n}. Put $p(x_n,\xi',\xi_n)=-i\xi_n|\xi|^{-2}+p'(x_n,\xi',\xi_n)$, where $p'(x,D)\in OPS^{-2}(\Omega)$. Since
\begin{equation}
\frac{1}{2\pi}\int^\infty_{-\infty}-i\xi_n|\xi|^{-2}d\xi_n=\Bigg\{\begin{array}{ll}
\frac{e^{-x_n|\xi'|}}{2} & x_n>0, \\
\frac{-e^{x_n|\xi'|}}{2} & x_n<0.\\
\end{array}
\end{equation}
Let $x_n$ goes to 0, the contribution of $p'$ will converge to the same limit from both positive and negative direction. Therefore, $v_\pm=Q_\pm f$, where $Q_\pm\in OPS^0(\partial M)$. $Q_\pm=\pm\frac{1}{2}I+Q'$, with $Q'\in OPS^{-1}(\partial M)$. Now, for $\partial_{\nu,y}\triangle_y\hat{E}(x,y)$, the expansion when $x$ is near $y$ is given by
\begin{equation}
\nabla_y\triangle_y\hat{E}(x,y)=-\frac{1}{\omega_{n-1}}d(x,y)^{1-n}V_{x,y}+\cdots,
\end{equation}
where $V_{x,y}$ denotes the unit vector at $y$ in the direction of the geodesic from $x$ to $y$. Therefore,
\begin{equation}
\partial_{\nu,y}\triangle_y\hat{E}(x,y)=-\frac{1}{\omega_{n-1}}d(x,y)^{1-n}\langle V_{x,y},\nu_y\rangle+\cdots.
\end{equation}
Since $\langle V_{x,y},\nu_y\rangle$ is Lipschitz on $\partial M\times\partial M$ and vanishes on the diagonal, $\partial_{\nu,y}\triangle_y\hat{E}(x,y)$ is integrable on $\partial M\times \partial M$. $Q_\pm$ has Schwartz kernels equal to $\partial_{\nu,y}\triangle_y\hat{E}(x,y)$ on the compliment of the diagonal in $\partial M\times \partial M$, together with the knowledge of the principle symbol of $Q_\pm$, we establish \eqref{n}.
\end{proof}

Now, we investigate the relation between the boundary biharmonic Steklov operators and the operators defined above. 

\begin{thm}
For the biharmonic Steklov operators, $\Theta$, $\Xi$ and $\Pi$, we have $\Xi$, $\Pi\in OPS^1(\partial M)$, $\Theta\in OPS^3(\partial M)$. All of them are elliptic. The principle symbol of $\Theta$ is equal to the principle symbol of $2\sqrt{-\triangle^T}^3$. The principle symbols of $\Xi$, $\Pi$ are equal to the principle symbol of $2\sqrt{-\triangle^T}$.
\end{thm}
\begin{proof}
For $f\in C^\infty(\partial M)$, let $u=K_1f\in C^\infty(\bar{M})$. Define operator $\theta$ on $\partial M$ to be $\theta f=\triangle u|_{\partial M}$. Since $\triangle(\triangle u)=0$, we have $\Lambda\theta f=-\Theta f$. From the Green's formula,
\begin{equation}
\begin{split}
u(x)&=\int_{\partial M}u\partial_{\nu,y}\triangle_y\hat{E}(x,y)-\partial_\nu u\triangle_y\hat{E}(x,y)+\triangle u\partial_{\nu,y}\hat{E}(x,y)-\partial_\nu\triangle u\hat{E}(x,y) d\sigma(y)\\
&=\int_{\partial M}f\partial_{\nu,y}\triangle_y\hat{E}(x,y)+\theta f\partial_{\nu,y}\hat{E}(x,y)-\Lambda\theta f\hat{E}(x,y) d\sigma(y)\\
&=L_4f(x)+L_2\theta f(x)-L_1\Lambda\theta f(x).
\end{split}
\end{equation}
for $x\in M$. Taking the limit, let $x$ goes to a boundary point, we have
\begin{equation}
f=\frac{1}{2}f+\frac{1}{2}Nf+S_2\theta f-S_3\Lambda\theta f
\end{equation}
on $\partial M$, which can be written as
\begin{equation}
(\frac{1}{2}-\frac{1}{2}N)f=(S_2-S_3\Lambda)\theta f.
\end{equation}
Note that $S_2$, $S_3\in OPS^{-3}(\partial M)$, $\Lambda\in OPS^1(\partial M)$, the principle symbols of $S_3$ and $\Lambda$ are given by $\frac{1}{4}|\xi'|^{-3}$, $|\xi'|$ respectively, we can conclude that $\theta\in OPS^2(\partial M)$ and the corresponding principle symbol is  $-2|\xi'|^2$. Therefore, $\Theta=-\Lambda\theta\in OPS^{3}(\partial M)$ with principle symbol $2|\xi'|^3$.

Now, we deal the operators $\Xi$ and $\Pi$ in a similar way. For $f\in C^\infty(\partial M)$, let $v=K_2f\in C^\infty(\bar{M})$. Using Green's formula,
\begin{equation}
\begin{split}
v(x)&=\int_{\partial M}v\partial_{\nu,y}\triangle_y\hat{E}(x,y)-\partial_\nu v\triangle_y\hat{E}(x,y)+\triangle v\partial_{\nu,y}\hat{E}(x,y)-\partial_\nu\triangle v\hat{E}(x,y) d\sigma(y)\\
&=\int_{\partial M}-f\triangle_y\hat{E}(x,y)+\Xi f\partial_{\nu,y} \hat{E}(x,y)-\Lambda\Xi f\hat{E}(x,y) d\sigma(y)\\
&=-L_3f(x)+L_2\Xi f(x)-L_1\Lambda\Xi f(x)
\end{split}
\end{equation}
for $x\in M$. Taking the limit as $x$ goes to a boundary point, we have
\begin{equation}
0=-S_1f+S_2\Xi f-S_3\Lambda\Xi f
\end{equation}
on $\partial M$, which is the same as
\begin{equation}
S_1f=(S_2-S_3\Lambda)\Xi f.
\end{equation}
Use the argument as above, we can conclude that $\Xi\in OPS^1(\partial M)$ with the principle symbol $2|\xi'|$. Finally, recall that $\Pi=\Xi+H$ on $\partial M$, we have $\Pi\in OPS^1(\partial M)$ with the same principle symbol.
\end{proof}

\begin{rmk}
The operator $\theta$ defined in the proof above may not be self-adjoint. In the proof, we only need the ellipticity of the operator.
\end{rmk}

For simplicity, in the following, we use $A\lesssim B$ to mean there exist constant $C$ independent of $\lambda$ such that $A\leq CB$ when $\lambda$ large enough. $A\approx B$ means $A\lesssim B$ and $B\lesssim A$.

One of the most important ingredients for the proof is the $L^p$ estimates for eigenfunctions. We have that $\sqrt[3]{\Theta}$, $\Xi$ and $\Pi$ are classical order 1 pseudo-differential operators with principle symbol equal to some nonzero constants times the principle symbol of $\sqrt{-\triangle^T}$. From \cite{S}, we have the following:
\begin{thm} \label{lp estimate}
For the Steklov eigenfunctions $e_\lambda$ satisfying \eqref{Theta}, $p\geq 2$, we have
\begin{equation} 
\|e_\lambda\|_{L^p(\partial M)}\lesssim \lambda^{\sigma(n,p)}\|e_\lambda\|_{L^2(\partial M)}.
\end{equation}
For the Steklov eigenfunctions $e_\lambda$ satisfying \eqref{Xi} or \eqref{Pi}, $p\geq 2$, we have
\begin{equation} 
\|\partial_\nu e_\lambda\|_{L^p(\partial M)}\lesssim \lambda^{\sigma(n,p)}\|\partial_\nu e_\lambda\|_{L^2(\partial M)}.
\end{equation}
where
\begin{equation}
\sigma(n,p)=
 \Bigg \{
\begin{array}{cc}
(n-1)(\frac{1}{2}-\frac{1}{p})-\frac{1}{2}, & \frac{2n}{n-2}\leq p\leq\infty \\
\frac{n-2}{2}(\frac{1}{2}-\frac{1}{p}), & 2\leq p\leq\frac{2n}{n-2}. \end{array}
\end{equation}
\end{thm}
Use this theorem for $p=\frac{2n}{n-2}$ and the Holder inequality, we have
\begin{equation}
\|e_\lambda\|_{L^1(\partial M)}\gtrsim\lambda^{-\frac{n-2}{4}}\|e_\lambda\|_{L^2(\partial M)}
\end{equation}
for $e_\lambda$ satisfying \eqref{Theta} and 
\begin{equation}
\|\partial_\nu e_\lambda\|_{L^1(\partial M)}\gtrsim\lambda^{-\frac{n-2}{4}}\|\partial_\nu e_\lambda\|_{L^2(\partial M)}
\end{equation}
for $e_\lambda$ satisfying \eqref{Xi}, \eqref{Pi}.

Now, we establish bounds of $L^p$ estimates when applying pseudo-differential operators to the eigenfunctions.
\begin{lem}
Fix $p\in(1,\infty)$, for any $P\in OPS^k(\partial M)$, we have
\begin{equation}
\|Pe_\lambda\|_{L^p(\partial M)}\lesssim\lambda^k\|e_\lambda\|_{L^p(\partial M)}
\end{equation}
for \eqref{Theta} and
\begin{equation}
\|P\partial_\nu e_\lambda\|_{L^p(\partial M)}\lesssim\lambda^k\|\partial_\nu e_\lambda\|_{L^p(\partial M)}
\end{equation}
for \eqref{Xi}, \eqref{Pi}.
\end{lem}
\begin{proof}
Let $e_\lambda$ satisfies \eqref{Theta}. Since the inverse of $I+\sqrt[3]{\Theta}$ exist, we have $P(I+\sqrt[3]{\Theta})^{-k}\in OPS^0(\partial M)$. Therefore,
\begin{equation}
\begin{split}
\|P e_\lambda\|_{L^p(\partial M)}&=\|P(I+\sqrt[3]{\Theta})^{-k}(I+\sqrt[3]{\Theta})^k e_\lambda\|_{L^p(\partial M)}\\
&=(1+\lambda)^k\|P(I+\sqrt[3]{\Theta})^{-k} e_\lambda\|_{L^p(\partial M)}\lesssim\lambda^k\|e_\lambda\|_{L^p(\partial M)}.
\end{split}
\end{equation}
We can get the similar result for $e_\lambda$ satisfying \eqref{Xi}, \eqref{Pi}.
\end{proof}
For the case that $p=1$, we need to take extra care.
\begin{lem}
Let $P\in OPS^k(\partial M)$. Fix $\epsilon>0$, $e_\lambda$ satisfying \eqref{Theta}, we have
\begin{equation}
\|Pe_\lambda\|_{L^1(\partial M)}\lesssim\lambda^{k+\epsilon}\|e_\lambda\|_{L^1(\partial M)}.
\end{equation}
If $e_\lambda$ satisfying \eqref{Xi}, \eqref{Pi}, similarly,
\begin{equation}
\|P\partial_\nu e_\lambda\|_{L^1(\partial M)}\lesssim\lambda^{k+\epsilon}\|\partial_\nu e_\lambda\|_{L^1(\partial M)}.
\end{equation}
\end{lem}
\begin{proof}
We proof the case for $k=0$ first. If $e_\lambda$ satisfies \eqref{Theta}, let $\delta>0$. By Holder's inequality,
\begin{equation}
\begin{split}
\|Pe_\lambda\|_{L^1(\partial M)}&\lesssim\|Pe_\lambda\|_{L^{1+\delta}(\partial M)}\lesssim\|e_\lambda\|_{L^{1+\delta}(\partial M)}\leq\|e_\lambda\|_{L^2(\partial M)}^\frac{2\delta}{1+\delta}\|e_\lambda\|_{L^1(\partial M)}^\frac{1-\delta}{1+\delta}\\
&\lesssim\lambda^\frac{2(n-1)\delta}{4(1+\delta)}\|e_\lambda\|_{L^1(\partial M)}.
\end{split}
\end{equation}
Choose $\delta$ such that $\frac{2(n-1)\delta}{4(1+\delta)}<\epsilon$, we can get the desired result. For general $k$, $P(I+\sqrt[3]{\Theta})^{-k}\in OPS^0(\partial M)$. We can use the same argument as in the lemma above. The case for \eqref{Xi}, \eqref{Pi} can be done in a similar manner.
\end{proof}

It's convenient to write the $L^p$ norms in terms of that of $\triangle e_\lambda$. We have the following corollary
\begin{cor}
Fix any $p\in[1,\infty)$. For \eqref{Theta}, we have
\begin{equation}
\|\triangle e_\lambda\|_{L^p(\partial M)}\approx \lambda^2\|e_\lambda\|_{L^p(\partial M)}.
\end{equation}
For \eqref{Pi}, we have
\begin{equation}
\|\triangle e_\lambda\|_{L^p(\partial M)}\approx \lambda\|\partial_\nu e_\lambda\|_{L^p(\partial M)}.
\end{equation}
\end{cor}
\begin{proof}
Choose $\epsilon=\frac{1}{2}$. For \eqref{Theta}, we have $\triangle e_\lambda|_{\partial M}=\theta e_\lambda$. Using $\theta+\sqrt[3]{2\Theta^2}=P_1'\in OPS^1(\partial M)$, we have
\begin{equation}
\theta e_\lambda=(-\sqrt[3]{2\Theta^2}+P_1')e_\lambda=(-\sqrt[3]{2}\lambda^2+P_1')e_\lambda
\end{equation}
on $\partial M$. Therefore, 
\begin{equation}
\sqrt[3]{2}\lambda^2\|e_\lambda\|_{L^p(\partial M)}-\|P_1'e_\lambda\|_{L^p(\partial M)}\leq\|\triangle e_\lambda\|_{L^p(\partial M)}\leq\sqrt[3]{2}\lambda^2\|e_\lambda\|_{L^p(\partial M)}+\|P_1'e_\lambda\|_{L^p(\partial M)}.
\end{equation}
Since $\|P_1'e_\lambda\|_{L^p(\partial M)}\lesssim \lambda^\frac{3}{2}\|e_\lambda\|_{L^p(\partial M)}$, we can get the desired result. The case for \eqref{Pi} is similar.
\end{proof}

\section{Lower bound for the vanishing set of $\triangle e_\lambda$}
For \eqref{Xi}, we can think $\triangle e_\lambda$ as the extension of the boundary data into $M$. Thus it would be interesting to get a lower bound of its vanishing set. Let 
%$Z_\lambda^\alpha=\{x\in M|e_\lambda=\alpha\}$,
$\hat{Z}_\lambda^\alpha=\{x\in M|\triangle e_\lambda=\alpha\}$ be the $\alpha$-level set of %$e_\lambda$,%
$\triangle e_\lambda$. Define 
\begin{equation}
\sigma_\alpha(x)=
\Bigg \{
\begin{array}{cc}
1 & x>\alpha \\
0 & x=\alpha \\
-1 & x<\alpha. \end{array}
\end{equation}
We have the following equation.
\begin{thm}\label{integratebyparts}
For any $f\in C^\infty(\bar{M})$, any regular value $\alpha$ of $\triangle e_\lambda$, we have

\begin{equation}
\int_{\partial M}f\sigma_\alpha(\triangle e_\lambda)\partial_\nu\triangle e_\lambda ds-\int_M\sigma_\alpha(\triangle e_\lambda)\langle\nabla f, \nabla\triangle e_\lambda\rangle dV=2\int_{\hat{Z}_\lambda^\alpha} f|\nabla\triangle e_\lambda|ds.
\end{equation}

\end{thm}
\begin{proof}
Let $\{\hat{D}_k^{+,\alpha}\}_k$ be the collection of connected components of the set $\{\triangle e_\lambda>\alpha\}\cap M$. $\hat{Z}_k^{+,\alpha}=\partial \hat{D}_k^{+,\alpha}\cap M$, $\hat{Y}_k^{+,\alpha}=\partial \hat{D}_k^{+,\alpha}\cap \partial M$. We have
\begin{equation}
\begin{split}
\int_{\hat{D}_k^{+,\alpha}}\langle\nabla f,\nabla\triangle e_\lambda\rangle dV&=-\int_{\hat{D}_k^{+,\alpha}}f\triangle^2 e_\lambda dV-\int_{\hat{Z}_k^{+,\alpha}} f|\nabla\triangle e_\lambda|ds+\int_{\hat{Y}_k^{+,\alpha}} f\partial_\nu\triangle e_\lambda ds\\
&=-\int_{\hat{Z}_k^{+,\alpha}} f|\nabla\triangle e_\lambda|ds+\int_{\hat{Y}_k^{+,\alpha}} f\partial_\nu\triangle e_\lambda ds.\\
\end{split}
\end{equation}
Similarly, from the set $\{\triangle e_\lambda<\alpha\}\cap M$, we can define $\hat{D}_k^{-,\alpha}$, $\hat{Z}_k^{-,\alpha}$, $\hat{Y}_k^{-,\alpha}$ together with a similar equation:
\begin{equation}
-\int_{\hat{D}_k^{-,\alpha}}\langle\nabla f,\nabla\triangle e_\lambda\rangle dV=-\int_{\hat{Z}_k^{-,\alpha}} f|\nabla\triangle e_\lambda|ds-\int_{\hat{Y}_k^{-,\alpha}} f\partial_\nu \triangle e_\lambda ds.\\
\end{equation}
Summing over all the equation above and notice that almost every point on $\hat{Z}_\lambda^\alpha$ will appear once for some $\hat{Z}_k^{+,\alpha}$ and once for some $\hat{Z}_k^{-,\alpha}$, we can get the desired equation.

\end{proof}

Plug in $f=1$ in the theorem, we get the following:

\begin{cor}
There exists a constant $c$ such that for any biharmonic Steklov eigenfunciton $e_\lambda$ satisfying \eqref{Theta}, \eqref{Xi}, \eqref{Pi}, any regular value $\alpha$ of $\triangle e_\lambda$ satisfying $|\alpha|<c\lambda^\frac{2-n}{4}\|\triangle e_\lambda\|_{L^2(\partial M)}$, we have

\begin{equation}
\int_{\hat{Z}_\lambda^\alpha} |\nabla\triangle e_\lambda|ds\gtrsim\lambda\|\triangle e_\lambda\|_{L^1(\partial M)}.
\end{equation}

\end{cor}
\begin{proof} For the eigenfunction satisfying \eqref{Xi}, we have

\begin{equation}
\begin{split}
2\int_{\hat{Z}_\lambda^\alpha} |\nabla\triangle e_\lambda|ds=&\int_{\partial M}\sigma_\alpha(\triangle e_\lambda)\partial_\nu\triangle e_\lambda ds
=\int_{\partial M}\sigma_\alpha(\triangle e_\lambda)\Lambda\triangle e_\lambda ds.
\end{split}
\end{equation}
Since $2\Lambda-\Xi=P\in OPS^0(\partial M)$, we have $\Lambda\triangle e_\lambda=\frac{1}{2}(\Xi+P)\triangle e_\lambda=\frac{1}{2}\lambda\triangle e_\lambda+\frac{1}{2}P\triangle e_\lambda$. Thus
\begin{equation} \label{Lowerbound}
\begin{split}
2\int_{\hat{Z}_\lambda^\alpha} |\nabla\triangle e_\lambda|ds=&\frac{1}{2}\int_{\partial M}\sigma_\alpha(\triangle e_\lambda)(\lambda+P)\triangle e_\lambda ds\\
=&\frac{\lambda}{2}\int_{\partial M}\sigma_\alpha(\triangle e_\lambda)\triangle e_\lambda ds+\frac{1}{2}\int_{\partial M}\sigma_\alpha(\triangle e_\lambda)P\triangle e_\lambda ds\\
\geq&\frac{\lambda}{2}\Big(\|\triangle e_\lambda\|_{L^1(\partial M)}-2\alpha|\partial M|\Big)-\frac{1}{2}\|P\triangle e_\lambda\|_{L^1(\partial M)}\\
\geq&\frac{\lambda}{2}\|\triangle e_\lambda\|_{L^1(\partial M)}-\lambda\alpha|\partial M|-\frac{C}{2}\lambda^\epsilon\|\triangle e_\lambda\|_{L^1(\partial M)}\\
\end{split}
\end{equation}
Choose $c$ which only depend on $M$ such that for any $|\alpha|<c\lambda^\frac{2-n}{4}\|\triangle e_\lambda\|_{L^2(\partial M)}$, $\alpha|\partial M|<\frac{1}{8}\|\triangle e_\lambda\|_{L^1(\partial M)}$. We can get the desired result when $\lambda$ is large.

For the eigenfunction satisfying \eqref{Pi}, use $\triangle e_\lambda=\Pi\partial_\nu e_\lambda+H\partial_\nu e_\lambda=\lambda\partial_\nu e_\lambda+H\partial_\nu e_\lambda$ and $2\Lambda-\Pi=P'\in OPS^0(\partial M)$. We have
\begin{equation}
\begin{split}
\Lambda\triangle e_\lambda&=\frac{1}{2}(\Pi+P')(\lambda\partial_\nu e_\lambda+H\partial_\nu e_\lambda)=\frac{1}{2}\lambda^2\partial_\nu e_\lambda+\frac{\lambda}{2}P'\partial_\nu e_\lambda+\frac{1}{2}\Lambda H\partial_\nu e_\lambda\\
&=\frac{1}{2}\lambda(\triangle e_\lambda-H\partial_\nu e_\lambda)+\frac{\lambda}{2}P'\partial_\nu e_\lambda+\frac{1}{2}\Lambda H\partial_\nu e_\lambda\\
&=\frac{\lambda}{2}\triangle e_\lambda+\frac{1}{2}(H\Pi+P'\Pi+\Lambda H)\partial_\nu e_\lambda.
\end{split}
\end{equation}
Notice that $\frac{1}{2}(H\Pi+P'\Pi+\Lambda H)\in OPS^1(\partial M)$. Therefore
\begin{equation}
\|\frac{1}{2}(H\Pi+P'\Pi+\Lambda H)\partial_\nu e_\lambda\|_{L^1(\partial M)}\lesssim\lambda^{1+\epsilon}\|\partial_\nu e_\lambda\|_{L^1(\partial M)}\lesssim\lambda^\epsilon\|\triangle e_\lambda\|_{L^1(\partial M)}.
\end{equation}
We can use the same approach as in \eqref{Lowerbound} to get the estimation for \eqref{Pi}.

For the eigenfunction satisfying \eqref{Theta}, we have $\triangle e_\lambda=\theta e_\lambda$ and $\partial_\nu\triangle e_\lambda=\Theta e_\lambda=-\lambda^3 e_\lambda$. $\sqrt[3]{2\Theta^2}+\theta=P_1'\in OPS^1(\partial M)$. We have
\begin{equation}
\begin{split}
\Lambda\triangle e_\lambda&=-\lambda^3 e_\lambda=-\lambda\sqrt[3]{\Theta^2}e_\lambda=\frac{\lambda\theta}{\sqrt[3]{2}}e_\lambda-\frac{\lambda P_1'}{\sqrt[3]{2}}e_\lambda=\frac{\lambda}{\sqrt[3]{2}}\triangle e_\lambda-\frac{\lambda P_1'}{\sqrt[3]{2}}e_\lambda,
\end{split}
\end{equation}
also,
\begin{equation}
\begin{split}
\|\frac{\lambda P_1'}{\sqrt[3]{2}}e_\lambda\|_{L^1(\partial M)}&\lesssim\lambda\cdot\lambda^{1+\epsilon}\|e_\lambda\|_{L^1(\partial M)}\lesssim\lambda^\epsilon\|\triangle e_\lambda\|_{L^1(\partial M)}.
\end{split}
\end{equation}
Again, we can use the same approach as in \eqref{Lowerbound} to get the estimation for \eqref{Theta} when $\lambda$ is sufficiently large.
\end{proof}

\begin{rmk}
For the operator $\Theta$, $\Pi$, the eigenfunctions are $e_\lambda|_{\partial M}$, $\partial_\nu e_\lambda$ respectively. It's more nature to write the norm in terms of the eigenfunctions. We choose $\triangle e_\lambda|_{\partial M}$ to make the result for all the cases look similar.
\end{rmk}

Next, we can plug in $f=\sqrt{1+|\nabla \triangle e_\lambda|^2}$ and get the following proposition.

\begin{prop}
For the eigenfunctions satisfying \eqref{Theta}, \eqref{Xi}, \eqref{Pi}, we have the following estimation when $\lambda$ is large enough:
\begin{equation}
\int_{\hat{Z}_\lambda^\alpha}|\nabla\triangle e_\lambda|^2 ds\lesssim \lambda^2\|\triangle e_\lambda\|_{L^2(\partial M)}^2.
\end{equation}
\end{prop}
\begin{proof}
Plug in $f=\sqrt{1+|\nabla\triangle e_\lambda|^2}$, we have
\begin{equation}
\begin{split}
2\int_{\hat{Z}_\lambda^\alpha}|\nabla\triangle e_\lambda|^2 ds&\leq 2\int_{\hat{Z}_\lambda^\alpha}|\nabla\triangle e_\lambda|\sqrt{1+|\nabla\triangle e_\lambda|^2} ds\\
&\leq\int_{\partial M}\sqrt{1+|\nabla\triangle e_\lambda|^2}|\partial_\nu\triangle e_\lambda|ds+\int_M|\langle\nabla \sqrt{1+|\nabla\triangle e_\lambda|^2}, \nabla\triangle e_\lambda\rangle| dV.
\end{split}
\end{equation}
We need to estimate both terms on the right hand side. First, let $e_\lambda$ satisfies \eqref{Xi}. $|\nabla\triangle e_\lambda|^2=(\partial_\nu\triangle e_\lambda)^2+|\nabla^T\triangle e_\lambda|^2$ on $\partial M$, where $\nabla^T$ denotes the gradient on $\partial M$. Since $\triangle^T\in OPS^2(\partial M)$, 
\begin{equation}
\begin{split}
\|\triangle^T\triangle e_\lambda\|_{L^2(\partial M)}=\|\triangle^T\Xi\partial_\nu e_\lambda\|_{L^2(\partial M)}\lesssim\lambda^3\|\partial_\nu e_\lambda\|_{L^2(\partial M)}\lesssim\lambda^2\|\triangle e_\lambda\|_{L^2(\partial M)}.
\end{split}
\end{equation}
We can get the following:
\begin{equation}
\begin{split}
\int_{\partial M}|\nabla^T\triangle e_\lambda|^2ds&=-\int_{\partial M} \triangle e_\lambda \triangle^T \triangle e_\lambda ds\leq\|\triangle e_\lambda\|_{L^2(\partial M)}\|\triangle^T\triangle e_\lambda\|_{L^2(\partial M)}\lesssim\lambda^2\|\triangle e_\lambda\|_{L^2(\partial M)}^2\\
\end{split}
\end{equation}
when $\lambda$ is large enough. Similarly, using $\Lambda\in OPS^1(\partial M)$, we have
\begin{equation}
\|\partial_\nu \triangle e_\lambda\|_{L^2(\partial M)}^2\lesssim\lambda^2\|\triangle e_\lambda\|_{L^2(\partial M)}^2,
\end{equation}
and therefore
\begin{equation}
\int_{\partial M}1+|\nabla\triangle e_\lambda|^2ds=\int_{\partial M}1+(\partial_\nu \triangle e_\lambda)^2+|\nabla^T\triangle e_\lambda|^2ds\lesssim \lambda^2\|\triangle e_\lambda\|_{L^2(\partial M)}^2.
\end{equation}
The estimation of the first term is given as
\begin{equation}
\int_{\partial M}\sqrt{1+|\nabla\triangle e_\lambda|^2}|\partial_\nu\triangle e_\lambda|ds\leq\|\sqrt{1+|\nabla\triangle e_\lambda|^2}\|_{L^2(\partial M)}\|\partial_\nu\triangle e_\lambda\|_{L^2(\partial M)}\lesssim \lambda^2\|\triangle e_\lambda\|_{L^2(\partial M)}^2.
\end{equation}
Now, let us estimate the second term.
\begin{equation}
\begin{split}
\int_M\big|\langle\nabla& \sqrt{1+|\nabla\triangle e_\lambda|^2}, \nabla\triangle e_\lambda\rangle\big|dV=\int_M\frac{|\nabla^2\triangle e_\lambda(\nabla\triangle e_\lambda,\nabla\triangle e_\lambda)|}{\sqrt{1+|\nabla\triangle e_\lambda|^2}}dV \\
&\leq\|\nabla^2\triangle e_\lambda\|_{L^2(M)}\|\nabla\triangle e_\lambda\|_{L^2(M)}\|\frac{\nabla\triangle e_\lambda}{\sqrt{1+|\nabla\triangle e_\lambda|^2}}\|_{L^\infty(M)}\\
&\leq\|\nabla^2\triangle e_\lambda\|_{L^2(M)}\|\nabla\triangle e_\lambda\|_{L^2(M)}.
\end{split}
\end{equation}
The $L^2$ norm of $\nabla\triangle e_\lambda$ and $\nabla^2\triangle e_\lambda$ on $M$ is needed. We have
\begin{equation}
\int_M|\nabla\triangle e_\lambda|^2dV=-\int_M\triangle e_\lambda\triangle^2 e_\lambda dV+\int_{\partial M}\triangle e_\lambda\partial_\nu\triangle e_\lambda ds\lesssim\lambda\|\triangle e_\lambda\|_{L^2(\partial M)}^2.
\end{equation}
Therefore, $\|\nabla\triangle e_\lambda\|_{L^2(M)}\lesssim\sqrt{\lambda}\|\triangle e_\lambda\|_{L^2(\partial M)}$. 

To estimate $\|\nabla^2\triangle e_\lambda\|_{L^2(M)}$, let us recall the Reilly's formula: for any smooth function $f$ on $M$, we have
\begin{equation}
\int_M |\nabla^2 f|^2+Ric(\nabla f,\nabla f)-(\triangle f)^2dV=\int_{\partial M} A(\nabla^Tf,\nabla^Tf)-2\partial_\nu f\triangle^T f+H(\partial_\nu f)^2ds.
\end{equation}
Use this formula for $\triangle e_\lambda$, we have
\begin{equation}
\begin{split}
\int_M |\nabla^2&\triangle e_\lambda|^2=-\int_MRic(\nabla\triangle e_\lambda,\nabla\triangle e_\lambda)dV\\
&+\int_{\partial M} A(\nabla^T\triangle e_\lambda,\nabla^T\triangle e_\lambda)-2\partial_\nu\triangle e_\lambda\triangle^T\triangle e_\lambda+H(\partial_\nu\triangle e_\lambda)^2ds\\
\leq& \|Ric\|_{L^\infty(M)}\|\nabla\triangle e_\lambda\|^2_{L^2(M)}+\|A\|_{L^\infty(\partial M)}\|\nabla^T\triangle e_\lambda\|^2_{L^2(\partial M)}\\
&+2\|\partial_\nu\triangle e_\lambda\|_{L^2(\partial M)}\|\triangle^T\triangle e_\lambda\|_{L^2(\partial M)}+\|H\|_{L^\infty(\partial M)}\|\partial_\nu\triangle e_\lambda\|_{L^2(\partial M)}^2\\
\lesssim&\Big(\|Ric\|_{L^\infty(M)}\cdot\lambda+(\|A\|_{L^\infty(\partial M)}+2\lambda)\cdot\lambda^2+\lambda^2\|H\|_{L^\infty(\partial M)}\Big)\|\triangle e_\lambda\|_{L^2(\partial M)}^2\\
\lesssim&\lambda^3\|\triangle e_\lambda\|_{L^2(\partial M)}^2.\\
\end{split}
\end{equation}
The estimation of the second term is given by
\begin{equation}
\begin{split}
\big|\int_M\langle\nabla \sqrt{1+|\nabla\triangle e_\lambda|^2}&, \nabla\triangle e_\lambda\rangle dV\big|\leq\|\nabla^2\triangle e_\lambda\|_{L^2(M)}\|\nabla\triangle e_\lambda\|_{L^2(M)}\\
&\lesssim\sqrt{\lambda}\sqrt{\lambda^3}\|\triangle e_\lambda\|_{L^2(\partial M)}^2=\lambda^2\|\triangle e_\lambda\|_{L^2(\partial M)}^2.
\end{split}
\end{equation}
Combine the estimations together, we have
\begin{equation}
\int_{\hat{Z}_\lambda}|\nabla\triangle e_\lambda|^2 ds\lesssim\lambda^2\|\triangle e_\lambda\|_{L^2(\partial M)}^2+\lambda^2\|\triangle e_\lambda\|_{L^2(\partial M)}^2\approx\lambda^2\|\triangle e_\lambda\|_{L^2(\partial M)}^2.
\end{equation}

For the eigenfunction satisfying \eqref{Pi}, just replace the operator $\Xi$ to $\Pi$ and we can get the desired result.

For the eigenfunction satisfying \eqref{Theta}, use similar method, we can get the following estimation on the boundary:
\begin{equation}
\begin{split}
\|\partial_\nu e_\lambda\|_{L^2(\partial M)}&=0,\|\nabla^T e_\lambda\|_{L^2(\partial M)}=\lambda\|e_\lambda\|_{L^2(\partial M)}\\
\|\triangle e_\lambda\|_{L^2(\partial M)}&, \|\triangle^T e_\lambda\|_{L^2(\partial M)}\lesssim \lambda^2\|e_\lambda\|_{L^2(\partial M)},\\
\|\nabla^T \triangle e_\lambda\|_{L^2(\partial M)}&\lesssim\lambda^3\|e_\lambda\|_{L^2(\partial M)}, \|\partial_\nu \triangle e_\lambda\|_{L^2(\partial M)}=\lambda^3\|e_\lambda\|_{L^2(\partial M)},\\
\|\triangle^T \triangle e_\lambda\|_{L^2(\partial M)}&\lesssim \lambda^4\| e_\lambda\|_{L^2(\partial M)},\\
\end{split}
\end{equation}
and following estimation on $M$:
\begin{equation}
\begin{split}
\|\nabla e_\lambda\|_{L^2(M)}&\approx\lambda^\frac{1}{2}\| e_\lambda\|_{L^2(\partial M)},\\
\|\triangle e_\lambda\|_{L^2(M)}&=\lambda^\frac{3}{2}\|e_\lambda\|_{L^2(\partial M)},\\
\|\nabla \triangle e_\lambda\|_{L^2(M)}&\lesssim\lambda^\frac{5}{2}\|e_\lambda\|_{L^2(\partial M)},\\
\|\nabla^2 \triangle e_\lambda\|_{L^2( M)}&\lesssim \lambda^\frac{7}{2}\| e_\lambda\|_{L^2(\partial M)}.\\
\end{split}
\end{equation}
From these, using $\|\triangle e_\lambda\|_{L^2(\partial M)}\approx\lambda^2\|e_\lambda\|_{L^2(\partial M)}$ when $\lambda$ is large, we can get the desired estimation.
\end{proof}
Finally, we can establish a lower bound of $|\hat{Z}_\lambda^\alpha|$.
\begin{thm}
For $e_\lambda$ satisfying \eqref{Theta}, \eqref{Xi} or \eqref{Pi}, for $|\alpha|<c\lambda^\frac{2-n}{4}\|\triangle e_\lambda\|_{L^2(M)}$, we have
\begin{equation}
\big|\hat{Z}_\lambda^\alpha\big|\gtrsim\lambda^{\frac{2-n}{2}}.
\end{equation}
\end{thm}
\begin{proof}
We have
\begin{equation}
\lambda\|\triangle e_\lambda\|_{L^1(\partial M)}\lesssim\int_{\hat{Z}_\lambda^\alpha} |\nabla\triangle e_\lambda|ds\leq\big(\int_{\hat{Z}_\lambda^\alpha} |\nabla\triangle e_\lambda|^2ds\big)^\frac{1}{2}\big|\hat{Z}_\lambda^\alpha\big|^\frac{1}{2}\lesssim \lambda\big|\hat{Z}_\lambda^\alpha\big|^\frac{1}{2}\|\triangle e_\lambda\|_{L^2(\partial M)}.
\end{equation}
Recall that when $\lambda$ is large, $p=1,2$, $\|\triangle e_\lambda\|_{L^p(\partial M)}\approx\lambda\|\partial_\nu e_\lambda\|_{L^p(\partial M)}$ for \eqref{Pi},  $\|\triangle e_\lambda\|_{L^p(\partial M)}\approx\lambda^2\|e_\lambda\|_{L^p(\partial M)}$ for \eqref{Theta}. Using the $L^p$ estimate \eqref{lp estimate} for the eigenfunctions, we have
\begin{equation}
\|\triangle e_\lambda\|_{L^1(\partial M)}\gtrsim\lambda^{-\frac{n-2}{4}}\|\triangle e_\lambda\|_{L^2(\partial M)}.
\end{equation}
Therefore,
\begin{equation}
\lambda^{\frac{2-n}{4}}\lesssim\big|\hat{Z}_\lambda^\alpha\big|^\frac{1}{2},
\end{equation}
which is the desired result.
\end{proof}
Plug in $\alpha=0$, we have the lower bound for the vanishing set of $\triangle e_\lambda$ as in theorem \ref{innerset}.

\section{Lower bound for the interior nodal set}

In this section, we get an lower bound for the interior nodal set. For problem \eqref{Xi}, \eqref{Pi}, the $\alpha$-level set is unstable near the boundary, since $e_\lambda$ vanishes on the boundary. For simplicity, we only consider the nodal set in this section. Let $Z_\lambda=\{x\in M|e_\lambda=0\}$ and $\sigma(x)=\sigma_0(x)$. 

We have the following equations.
\begin{thm}
For the problem \eqref{Theta}, let $f\in C^\infty(\bar{M})$, if 0 is a regular value of $e_\lambda$, we have
\begin{equation}
\int_{\partial M}f\sigma(e_\lambda)\partial_\nu\triangle e_\lambda ds-\int_M\sigma(e_\lambda)\langle\nabla f, \nabla\triangle e_\lambda\rangle dV=2\int_{Z_\lambda} f\langle\nabla\triangle e_\lambda,N\rangle ds.
\end{equation}
For the problem \eqref{Xi}, \eqref{Pi}, if 0 is a regular value of $e_\lambda$, we have
\begin{equation}
-\int_{\partial M}f\sigma(\partial_\nu e_\lambda)\partial_\nu\triangle e_\lambda ds-\int_M\sigma(e_\lambda)\langle\nabla f, \nabla\triangle e_\lambda\rangle dV=2\int_{Z_\lambda} f\langle\nabla\triangle e_\lambda,N\rangle ds,
\end{equation}
where $N$ on $Z_\lambda$ is defined to be the unit normal $\frac{\nabla e_\lambda}{|\nabla e_\lambda|}$.
\end{thm}
\begin{proof}
The result follows by replacing $\{D_k^{+,\alpha}\}_k$ to be the collection of connected components of the set $\{\triangle e_\lambda>0\}$ in the Theorem \ref{integratebyparts}.
\end{proof}

Plug in $f=1$ in the theorem, we get the following:
\begin{cor}
There exists a constant $c$ such that for any biharmonic Steklov eigenfunciton $e_\lambda$ satisfying \eqref{Theta}, with 0 as a regular value, we have
\begin{equation}
\int_{Z_\lambda} |\langle\nabla\triangle e_\lambda,N\rangle|ds\geq\frac{\lambda^3}{2}\|e_\lambda\|_{L^1(\partial M)}.
\end{equation}
For any eigenfunction satisfying \eqref{Pi} or \eqref{Xi}, with 0 as a regular value, we have
\begin{equation}\label{lowerbound}
\int_{Z_\lambda} |\langle\nabla\triangle e_\lambda,N\rangle|ds\gtrsim\lambda^2\|\partial_\nu e_\lambda\|_{L^1(\partial M)}.
\end{equation}
\end{cor}
\begin{proof} For the eigenfunction satisfying \eqref{Theta}, we have
\begin{equation}
\begin{split}
2\int_{Z_\lambda} |\langle\nabla\triangle e_\lambda,N\rangle|ds\geq&-2\int_{Z_\lambda} \langle\nabla\triangle e_\lambda, N\rangle ds=-\int_{\partial M}\sigma(e_\lambda)\partial_\nu\triangle e_\lambda ds\\
=&\int_{\partial M}\sigma(e_\lambda)\lambda^3 e_\lambda ds=\lambda^3\int_{\partial M}|e_\lambda| ds=\lambda^3 \|e_\lambda\|_{L^1(\partial M)}.
\end{split}
\end{equation}
For the eigenfunction of satisfying \eqref{Xi} or \eqref{Pi}, we have
\begin{equation}
\begin{split}
2\int_{Z_\lambda} |\langle\nabla\triangle e_\lambda,N\rangle|ds\geq&-2\int_{Z_\lambda} \langle\nabla\triangle e_\lambda, N\rangle ds=\int_{\partial M}\sigma(\partial_\nu e_\lambda)\partial_\nu\triangle e_\lambda ds\\
\gtrsim&\lambda^2 \|\partial_\nu e_\lambda\|_{L^1(\partial M)}.
\end{split}
\end{equation}
\end{proof}

Now, we need to get an upper bound for $|\nabla\triangle e_\lambda|$. The approach is the same as that in Proposition 3.1 of \cite{SWZ}: Applying the gradient estimates of elliptic equation in the interior and near the boundary separately and combine the results.

\begin{prop}
If $e_\lambda$ satisfies \eqref{Xi} or \eqref{Pi}, $d=d(x)$ be the distance from $x\in M$ to $\partial M$, we have
\begin{equation} \label{upperbound}
\|(\lambda^{-1}+d)\nabla\triangle e_\lambda\|_{L^\infty(M)}\lesssim \lambda^\frac{n}{2}\|\partial_\nu e_\lambda\|_{L^1(\partial M)}.
\end{equation}
\end{prop}
\begin{proof}
On the boundary, $\triangle e_\lambda=\lambda\partial_\nu e_\lambda$ for problem \eqref{Xi} and $\triangle e_\lambda=\lambda\partial_\nu e_\lambda+H\partial_\nu e_\lambda$ for problem \eqref{Pi}. We can argue as in \cite{SWZ}, see also \cite{SZ} that
\begin{equation}
\lambda^{-k}\|(D^T)^k\triangle e_\lambda\|_{L^\infty(\partial M)}\lesssim \lambda^\frac{n}{2}\|\partial_\nu e_\lambda\|_{L^1(\partial M)},
\end{equation}
where $(D^T)^k$ denotes $k$ boundary derivatives.

For the interior estimate, start with
\begin{equation}
\|\triangle e_\lambda\|_{L^\infty(\partial M)}\lesssim \lambda^\frac{n}{2}\|\partial_\nu e_\lambda\|_{L^1(\partial M)},
\end{equation}
since $\triangle e_\lambda$ is harmonic, from the gradient estimate, see corollary 6.3 of \cite{GT}, for a fixed $\delta>0$,
\begin{equation}
\|d\nabla\triangle e_\lambda\|_{L^\infty(\{d\geq \delta\lambda^{-1}\})}\leq C_\delta \lambda^\frac{n}{2}\|\partial_\nu e_\lambda\|_{L^1(\partial M)}.
\end{equation}
The constant $C_\delta$ depends on $\delta$ and $M$, but not on $\lambda$.

Now, for the boundary estimate for any $x_0\in\partial M$, use a local coordinate in a neighborhood of $x_0$ which map $x_0$ to 0, $\partial M$ to $\{x_n=0\}$, and the neighborhood of $x_0$ into the upper half space. For simplicity, we also us $e_\lambda$ to denote the function induced on the coordinate. Consider the ball of radius $2\delta\lambda^{-1}$ around $0$ and define 
\begin{equation}
u_\lambda(x)=\lambda^{-\frac{n}{2}}\triangle_M e_\lambda(x\lambda^{-1}),
\end{equation}
which is defined in the upper half of the ball of radius $2\delta$, $B^+_{2\delta}(0)$. We have the estimate
\begin{equation}
\|(D^T)^k u_\lambda\|_{L^\infty(\mathbb{R}^{n-1}\cap B^+_{2\delta}(0))}\leq C_k \|\partial_\nu e_\lambda\|_{L^1(\partial M)}.
\end{equation}
The partial differential equation satisfied by $u$ has uniformly bounded coefficients. We can also find $\phi_\lambda$ in this coordinate which agree with $u_\lambda$ on the boundary and is bounded in $C^{2,\alpha}({B^+_{2\delta}(0)})$ by some constant times $\|\partial_\nu e_\lambda\|_{L^1(\partial M)}$. Use Corollary 8.36 in \cite{GT}, the $C^{1,\alpha}({B^+_{\delta}(0)})$ norm of $u_\lambda$ is bounded by $C_\alpha \|\partial_\nu e_\lambda\|_{L^1(\partial M)}$, with $C_\alpha$ independent of $\lambda$. Thus, we have
\begin{equation}
\|Du_\lambda\|_{L^\infty(B^+_{\delta}(0))}\leq C_\alpha \|\partial_\nu e_\lambda\|_{L^1(\partial M)},
\end{equation}
which is the desired result.
\end{proof}

\begin{prop}
If $e_\lambda$ satisfies \eqref{Theta}, we have
\begin{equation}
\|(\lambda^{-1}+d)\nabla\triangle e_\lambda\|_{L^\infty(M)}\lesssim \lambda^\frac{n+4}{2}\|e_\lambda\|_{L^1(\partial M)}.
\end{equation}
\end{prop}
\begin{proof}
On the boundary, $\partial_\nu\triangle e_\lambda=-\lambda^3 e_\lambda$ for problem \eqref{Theta}, we have that
\begin{equation}
\lambda^{-k}\|(D^T)^k\partial_\nu\triangle e_\lambda\|_{L^\infty(\partial M)}\lesssim \lambda^\frac{n+4}{2}\|\partial_\nu e_\lambda\|_{L^1(\partial M)},
\end{equation}
where $(D^T)^\alpha$ denotes $\alpha$ boundary derivatives.

For the interior estimate, us the result in \cite{YZ}, we have
\begin{equation}
\|\triangle e_\lambda\|_{L^\infty(\partial M)}\leq C\|\partial_\nu\triangle e_\lambda\|_{L^\infty(\partial M)}\lesssim \lambda^\frac{n+4}{2}\|e_\lambda\|_{L^1(\partial M)}
\end{equation}
and therefore for any given $\delta>0$,
\begin{equation}
\|d\nabla\triangle e_\lambda\|_{L^\infty(\{d\geq \delta\lambda^{-1}\})}\leq C_\delta \lambda^\frac{n+4}{2}\|e_\lambda\|_{L^1(\partial M)}.
\end{equation}

Now, for the boundary estimate, for any $x_0\in\partial M$, use the same approach as above, define 
\begin{equation}
u_\lambda(x)=\lambda^{-\frac{n+4}{2}}\triangle e_\lambda(x\lambda^{-1}),
\end{equation}
which is defined in $B^+_{2\delta}(0)$. We have the estimate
\begin{equation}
\|(D^T)^k \partial_\nu u_\lambda\|_{L^\infty(\partial M)}\leq C_k \|e_\lambda\|_{L^1(\partial M)}.
\end{equation}
From lemma 6.29 in \cite{GT}, we have the following bound:
\begin{equation}
\|u_\lambda\|_{C^{2,\alpha}}\leq C(\|u_\lambda\|_{L^\infty}+\|\partial_\nu u_\lambda\|_{C^{1,\alpha}}).
\end{equation}
Thus, we have
\begin{equation}
\|Du_\lambda\|_{L^\infty(B^+_{\delta}(0))}\leq C_\alpha \|e_\lambda\|_{L^1(\partial M)},
\end{equation}
and therefore
\begin{equation}
\|\lambda^{-1}\nabla\triangle e_\lambda\|_{L^\infty(B^+_{\delta}(0))}\leq C \lambda^\frac{n+4}{2} \|e_\lambda\|_{L^1(\partial M)}.
\end{equation}
\end{proof}

Now, we can estimate the interior set in each case.

\begin{proof}[Proof of Theorem \ref{realinnerset}]
For problem \eqref{Xi}, \eqref{Pi}, we have \eqref{lowerbound} and \eqref{upperbound}. Therefore
\begin{equation}
\begin{split}
\lambda^2\|\partial_\nu e_\lambda\|_{L^1(\partial M)}&\lesssim\int_{Z_\lambda} |\langle\nabla\triangle e_\lambda,N\rangle|ds\leq\|\nabla\triangle e_\lambda\|_{L^\infty(M)}|Z_\lambda|\\
&\lesssim\lambda^{\frac{n+2}{2}}\|\partial_\nu e_\lambda\|_{L^1(\partial M)}|Z_\lambda|.
\end{split}
\end{equation}
Cancel $\|\partial_\nu e_\lambda\|_{L^1(\partial M)}$ from the both side yields the desired result.

For problem \eqref{Theta}, we can use a similar argument.
\end{proof}
\begin{rmk}
For problem \eqref{Theta}, we can not get the $L^\infty$ bound of $\triangle e_\lambda$ on the boundary. We use the $L^\infty$ bound of $\partial_\nu\triangle e_\lambda$ instead, thus losing a factor of $\lambda$.
\end{rmk}
\section{Lower bound for the boundary nodal set }
Let us turn our attention to the boundary $\partial M$ and get the estimations of the nodal sets for the operators $\Theta$, $\Xi$ and $\Pi$. Since all we need is the property of the operator on $\partial M$, we can argue in an abstract way. Let $\Psi\in OPS^1(\partial M)$ be classical and with the principle symbol equal to some nonzero constant times the principle symbol of $\sqrt{-\triangle^T}$. Let $\phi_\lambda$ be an eigenfunction of $\Psi$ corresponds to $\lambda$. Note that the case we want is given by $\Psi=\sqrt[3]{\Theta}$, $\Xi$, $\Pi$ and $\phi_\lambda=e_\lambda|_{\partial M}$, $\partial_\nu e_\lambda$, $\partial_\nu e_\lambda$ respectively.

The proof is given in \cite{WZ} to establish the lower bound of boundary nodal sets of harmonic Steklov eigenfunctions. From now on, all the argument are on $\partial M$ and all the $L^p$ norm are $L^p(\partial M)$. Let $\tilde{Z}_\lambda^\alpha=\{x\in\partial M|\phi_\lambda=\alpha\}$ be the $\alpha$-level set of 
$\phi_\lambda$. We have the following equation.
\begin{thm}
For any $f\in C^\infty(\bar{M})$, any regular value $\alpha$ of $\phi_\lambda$, we have
\begin{equation}
-\int_{\partial M}\sigma_\alpha(\phi_\lambda)\Big[\langle\nabla^Tf,\nabla^T\phi_\lambda\rangle+f\triangle^T\phi_\lambda
\Big]dV=2\int_{\tilde{Z}_\lambda^\alpha} f|\nabla^T\phi_\lambda|ds.
\end{equation}
\end{thm}
\begin{proof}
Let $\{\tilde{D}_k^{+,\alpha}\}_k$ be the collection of  connected components of the set $\{\phi_\lambda>\alpha\}$, $\tilde{Z}_k^{+,\alpha}=\partial \tilde{D}_k^{+,\alpha}$.  we have
\begin{equation}
-\int_{\tilde{D}_k^{+,\alpha}}\langle\nabla^T f,\nabla^T\phi_\lambda\rangle+f\triangle^T\phi_\lambda dV=\int_{\tilde{Z}_k^{+,\alpha}} f|\nabla^T\phi_\lambda|ds.
\end{equation}
Similarly, from the set $\{\phi_\lambda<\alpha\}$, we can define $\tilde{D}_k^{-,\alpha}$, $\tilde{Z}_k^{-,\alpha}$ together with a similar equation:
\begin{equation}
\int_{\tilde{D}_k^{-,\alpha}}\langle\nabla^T f,\nabla^T\phi_\lambda\rangle+f\triangle^T\phi_\lambda dV=\int_{\tilde{Z}_k^{-,\alpha}} f|\nabla^T\phi_\lambda|ds.
\end{equation}
Summing over all the equations and we can get the desired equation.
\end{proof}

Choosing $f=1$ gives the following:
\begin{cor}
There exists a constant $c$ such that for any regular value $\alpha$ of $\phi_\lambda$ satisfying $|\alpha|<c\lambda^\frac{2-n}{4}\|\phi_\lambda\|_{L^2}$, we have
\begin{equation}
\int_{\tilde{Z}_\lambda^\alpha} |\nabla^T\phi_\lambda|ds\gtrsim\lambda^2\|\phi_\lambda\|_{L^1}.
\end{equation}
\end{cor}
\begin{proof} Put $f=1$ yields
\begin{equation}
2\int_{\tilde{Z}_\lambda^\alpha}|\nabla^T\phi_\lambda|ds=-\int_{\partial M}\sigma_\alpha(\phi_\lambda)\triangle^T\phi_\lambda
dV.
\end{equation}
Since $\sqrt{-\triangle^T}=a\Psi+P_0$, for some $a\neq0$, $P_0\in OPS^0(\partial M)$, 
\begin{equation}
\triangle^T\phi_\lambda=-a^2\Psi^2\phi_\lambda-(a\Psi P_0+aP_0 \Psi+P_0^2)\phi\lambda=-a^2\lambda^2\phi_\lambda-(a\Psi P_0+aP_0 \Psi+P_0^2)\phi_\lambda.
\end{equation}
Using $a\Psi P_0+aP_0 \Psi+P_0^2\in OPS^1(\partial M)$, we can bound the second term by
\begin{equation}
\|(a\Psi P_0+aP_0 \Psi+P_0^2)\phi_\lambda\|_{L^1}\lesssim\lambda^{1+\epsilon}\|\phi_\lambda\|_{L^1}.
\end{equation}
Proceed as in \eqref{Lowerbound}, and choose the constant $c$ as before, we can get the desired result.
\end{proof}

Next, choosing $f=\sqrt{1+|\nabla^T \phi_\lambda|^2}$ gives the following proposition.
\begin{prop}
We have the following estimation when $\lambda$ large enough:
\begin{equation}
\int_{\tilde{Z}_\lambda^\alpha}|\nabla^T\phi_\lambda|^2 ds\lesssim \lambda^3\|\phi_\lambda\|_{L^2}^2.
\end{equation}
\end{prop}
\begin{proof}
Plug in $f=\sqrt{1+|\nabla^T\phi_\lambda|^2}$,
\begin{equation}
\begin{split}
2\int_{\tilde{Z}_\lambda^\alpha}|\nabla^T\triangle e_\lambda|^2 ds&\leq 2\int_{\tilde{Z}_\lambda^\alpha}|\nabla^T\phi_\lambda|\sqrt{1+|\nabla^T\phi_\lambda|^2} ds\\
&\leq\int_{\partial M}\sqrt{1+|\nabla^T\phi_\lambda|^2}|\triangle^T\phi_\lambda|+|\langle\nabla^T \sqrt{1+|\nabla^T\phi_\lambda|^2}, \nabla^T\phi_\lambda\rangle| dV.
\end{split}
\end{equation}

Since $\triangle^T\in OPS^2(\partial M)$, we can use the lemma for $L^p$ bounds to get
\begin{equation}
\|\triangle^T\phi_\lambda\|_{L^2}\lesssim\lambda^2\|\phi_\lambda\|_{L^2}
\end{equation}
and
\begin{equation}
\int_{\partial M}\|\nabla^T\phi_\lambda\|^2=-\int_{\partial M}\phi_\lambda\triangle^T\phi_\lambda\leq\|\phi_\lambda\|_{L^2}\|\triangle^T\phi_\lambda\|_{L^2}\lesssim\lambda^2\|\phi_\lambda\|_{L^2}^2.
\end{equation}
Therefore, the first term is bounded by
\begin{equation}
\int_{\partial M}\sqrt{1+|\nabla^T\phi_\lambda|^2}|\triangle^T\phi_\lambda|dV\leq\|\sqrt{1+|\nabla^T\phi_\lambda|^2}\|_{L^2}\|\triangle^T\phi_\lambda\|_{L^2}\lesssim\lambda^3\|\phi_\lambda\|_{L^2}^2.
\end{equation}

For the the second term,
\begin{equation}
\begin{split}
\int_{\partial M}\big|\langle\nabla^T& \sqrt{1+|\nabla^T\phi_\lambda|^2}, \nabla^T\phi_\lambda\rangle\big|dV=\int_{\partial M}\frac{|(\nabla^T)^2\phi_\lambda(\nabla^T\phi_\lambda,\nabla^T\phi_\lambda)|}{\sqrt{1+|\nabla^T\phi_\lambda|^2}}dV \\
&\leq\|(\nabla^T)^2\phi_\lambda\|_{L^2}\|\nabla^T\phi_\lambda\|_{L^2}\|\frac{\nabla^T\phi_\lambda}{\sqrt{1+|\nabla^T\phi_\lambda|^2}}\|_{L^\infty}\lesssim\lambda\|(\nabla^T)^2\phi_\lambda\|_{L^2}\|\phi_\lambda\|_{L^2}.
\end{split}
\end{equation}
Since $\partial M$ is compact without boundary, for any smooth function $f$ on $\partial M$,
\begin{equation}
\int_{\partial M} |(\nabla^T)^2 f|^2=-\int_{\partial M}Ric_{\partial M}(\nabla^T f,\nabla^T f)+(\triangle^T f)^2dV.
\end{equation}
Use this formula on $\phi_\lambda$,
\begin{equation}
\int_{\partial M} |(\nabla^T)^2\phi_\lambda|^2\leq\|Ric_{\partial M}\|_{L^\infty}\|\nabla^T\phi_\lambda\|_{L^2}^2+\|\triangle^T\phi_\lambda\|_{L^2}^2\lesssim\lambda^4\|\phi_\lambda\|_{L^2}^2.
\end{equation}
Thus the second term is bounded by
\begin{equation}
\int_{\partial M}|\langle\nabla^T \sqrt{1+|\nabla^T\phi_\lambda|^2}, \nabla^T\phi_\lambda\rangle|dV\lesssim\lambda\|(\nabla^T)^2\phi_\lambda\|_{L^2}\|\phi_\lambda\|_{L^2}\lesssim\lambda^3\|\phi_\lambda\|_{L^2}^2
\end{equation}
Combining the estimation for both term, we can get the desired bound for $\int_{\tilde{Z}_\lambda^\alpha}|\nabla^T\phi_\lambda|^2 ds$.
\end{proof}
Finally, we can estimate the size of boundary nodal sets.
\begin{thm}
For $\phi_\lambda$, $\alpha$ as above, we have
\begin{equation}
\big|\tilde{Z}_\lambda^\alpha\big|\gtrsim\lambda^{\frac{4-n}{2}}.
\end{equation}
\end{thm}
\begin{proof}
From the bounds above, 
\begin{equation}
\lambda^2\|\phi_\lambda\|_{L^1}\lesssim\int_{\tilde{Z}_\lambda^\alpha} |\nabla^T\phi_\lambda|ds\leq\big(\int_{\tilde{Z}_\lambda^\alpha} |\nabla^T\phi_\lambda|^2ds\big)^\frac{1}{2}\big|\tilde{Z}_\lambda^\alpha\big|^\frac{1}{2}\lesssim \lambda^\frac{3}{2}\big|\tilde{Z}_\lambda^\alpha\big|^\frac{1}{2}\|\phi_\lambda\|_{L^2}.
\end{equation}
Using the $L^p$ estimate \eqref{lp estimate} for the $\phi_\lambda$,
\begin{equation}
\|\phi_\lambda\|_{L^1}\gtrsim\lambda^{-\frac{n-2}{4}}\|\phi_\lambda\|_{L^2}.
\end{equation}
Therefore,
\begin{equation}
\lambda^{\frac{4-n}{4}}\lesssim\big|\hat{Z}_\lambda^\alpha\big|^\frac{1}{2}.
\end{equation}
\end{proof}
We can get the theorem \ref{boundaryset} by plugging in $\alpha=0$.

\end{document}